\title{The continuity of  curvature flow 
in Fr\'echet distance 
}
\author{Shiyi Ma
        }
\date{}
\documentclass{article}
\usepackage{microtype} %% Microtypographic features of PDFLaTeX

\usepackage{placeins}

\usepackage{titlesec}

\usepackage{graphicx}

\usepackage[utf8]{inputenc}
\usepackage{amsmath}
\usepackage{amsfonts, amssymb, amsthm}
\usepackage{animate}
\usepackage{bm}
\usepackage{breqn}
\usepackage{calc}
\usepackage{caption}
\usepackage{csquotes}
\usepackage{cite}
\usepackage{color}
\usepackage{fancyhdr}
\usepackage{fancyvrb}
\usepackage{float}
\usepackage[margin=1.25in]{geometry}
\usepackage{graphicx}
\usepackage{lineno}
\usepackage{listings}
\usepackage{setspace}
\usepackage{stackengine}
\usepackage{scalerel}
\usepackage{tcolorbox}
\usepackage{wrapfig}
\usepackage{blindtext}
\usepackage{lmodern}

\newtheorem{theorem}{Theorem}[section]
\newtheorem{lemma}{Lemma}[section]

\usepackage{amsmath}
\usepackage{amsthm}
\usepackage{setspace}

\usepackage{lipsum} % for mock text

\usepackage{hyperref}
\hypersetup{
  colorlinks,
  citecolor=red,
  linkcolor=red,
  urlcolor=blue}

\usepackage{cleveref}

\usepackage{tikz-cd}

\usepackage{tikz-cd}
\usetikzlibrary{cd}
\newcommand{\R}{{\mathbb R}}

\newcommand{\sph}{{\mathbf S}}

\DeclareMathOperator{\dist}{dist}

\DeclareMathOperator{\cf}{cf}

\DeclareMathOperator{\tr}{int}

\newcommand{\interior}[1]{%
  {\kern0pt#1}^{\mathrm{o}}%
}

\titlespacing{\section}{1pt}{*1}{*1}
\titlespacing{\subsection}{1pt}{*1}{*1}

\begin{document}
%\newpage
\maketitle
%%%%%%%%%%%%%%%%%%%%%%%%%%%
% abstract, keywords and Subject classification are optional.
%%%%%%%%%%%%%%%%%%%%%%%%%%%

\begin{abstract}
We study the evolution of a Jordan curve on the plane by curvature flow, also known as curve shortening flow, and by level-set flow, which is a weak formulation of curvature flow.
We show that the evolution of the  curve depends continuously on the initial curve in Fr\'echet distance. This is an extension of Michael Dobbins'  work on the evolution of Jordan curves bisecting  the $2$-sphere  by curvature flow.

    \end{abstract}

%\replace{$\cf$}{$cf$}

% Most people don't use these, so they are "commented out"
% by starting the lines with a "%"
%\begin{keywords}
%   \LaTeX, typesetting
%\end{keywords}

%\begin{AMS}
%   50C60, 18C25
%\end{AMS}

%%%%%%%%%%%%%%%%%%%%%%
% % Here is the start of the Text
%%%%%%%%%%%%%%%%%%%%%%
\section{Introduction}
\noindent
Let $\mathfrak{J}$ denote the set of Jordan curves (which are also called simple closed curves) with zero Lebesgue measure, on $\R^2$. 
Michael Gage and Richard Hamilton proved that all smooth convex curves eventually contract to a point without forming any other singularities \cite{Gage}, and Matthew Grayson proved that every non-convex curve will eventually become convex \cite{Grayson}. 
Let $\mathfrak{J}_T$ be a subset of $\mathfrak{J}$ for which the curvature flow problem has a solution up to time $T$, for $ 0 < T < \infty$. 
%where the  lifetime of each Jordan curve is  strictly greater than  $T$,  for $ 0 < T < \infty$. %Therefore,   $\mathfrak{J}_T \subset \mathfrak{J}$. 
Define a function $\cf: \mathfrak{J}_T \times [0, T)_{\R} \to \mathfrak{J}$ by evolution by curvature flow. The main result of this article is that this function is continuous in Fr\'echet distance.

\begin{theorem}\label{1.1}
    Let $\gamma_k \in \mathfrak{J}_T$, and $t_k \in [0, T)_{\R}$ for $k \in \{1, ..., \infty\}$. If $\gamma_k \to \gamma_{\infty}$ in Fr\'echet distance and $t_k \to t_{\infty}$, then $\cf(\gamma_k, t_k) \to \cf(\gamma_{\infty}, t_{\infty})$ in Fr\'echet distance. 
\end{theorem}

\noindent
 The challenge of proving Theorem \ref{1.1} is to prove Lemma \ref{Existence} where given a curve $\gamma$, we need to show the existence of a curve that always intersects $\gamma$ at exactly two points until $\gamma$ shrinks to a point. To prove Lemma \ref{Existence},  we construct a family of Jordan curves through homotopy, and prove that  curves in this family instantly evolve into three classes, among which one  class contains  curves that intersect $\gamma$ at exactly two points throughout the lifetime of $\gamma$.

 \subsection{Definitions and notations}
 \noindent
We denote real intervals by $(a, b]_{\R} = \{x \in  R : a < x \leq b \} $ for bounds $ a, b \in \R $ with
any combination of round or square brackets for (half) open or closed intervals. We also similarly denote segments in $\R^2$ by $[a, b]_{\R^2}= \{ ta +(1-t)b: t \in [0, 1]_{\R}\}$ for $a, b \in \R^2$. 
We denote the unit circle in $\R^2$ by $\mathbf{S}^1$. Given a set $A \subset \R^2$, we denote the interior of $A$ by $\tr  A$. 
For $X \subseteq \R^2$ and $\delta > 0$, let $X \oplus \delta$ be the set of points at most distance $\delta$ from $X$. Every Jordan curve on $\R^2$ is treated as a subset of $\R^2$.
%For a Jordan curve $\gamma$, we denote the closed region enclosed by $\gamma$ by $\mathcal{I}(\gamma)$, and its  complement, $\R^2 \setminus \mathcal{I}(\gamma)$, by $\mathcal{O}(\gamma)$. 
Note that for every Jordan curve $\gamma$, $\R^2 \setminus \gamma$ consists of two components. We denote 
the  bounded component, which is the interior of the region enclosed by $\gamma$ by $\mathcal{I}(\gamma)$,  and  the other  unbounded component by $\mathcal{O}(\gamma)$. 
%We denote the bounded component by $\mathcal{B}(\gamma)$ and the unbounded component by $\mathcal{U}(\gamma)$. 
 Let $\gamma_1, \gamma_2$ be two Jordan curves on $\R^2$. The Hausdorff distance is defined by $$\dist_H(\gamma_1, \gamma_2) = \inf\{ \delta \ | \ \gamma_1 \subseteq \gamma_2 \oplus \delta, \gamma_2 \subseteq \gamma_1 \oplus \delta\}.$$
The Fr\'echet distance on $\mathfrak{J}$ is defined by $$\dist_{F}(\gamma_1, \gamma_2) = \inf \limits_{\varphi_1, \varphi_2} \sup \limits_{x} ||\varphi_1(x) - \varphi_2(x)||,$$
where $\varphi_i: \sph^1 \to \gamma_i$ are homeomorphisms.
Let $\mathfrak{J}_{c^{\infty}}$ be the set of smooth Jordan curves. Given a smooth curve $\gamma \in \mathfrak{J}_{C^{\infty}}$ and a point $p \in \gamma$, let $\mathrm{kn} ( \gamma, p) \in \R^2$ denote the vector of  curvature of $\gamma$ at $p$. A solution to the curvature flow problem for a given initial curve $\gamma_0 \in \mathfrak{J}$ and stopping time $T \in (0, \infty)_{\R}$, is  a map $\Gamma : \sph^1 \times (0, T) \to \R^2$ such that $\partial_t \Gamma(x, t) = \mathrm{kn}(\Gamma(\sph^1, t) , \Gamma(x, t))$ and $\Gamma(\sph^1, t) \to \gamma_0$ in Fr\'echet distance as $t \to 0$. Let $$ \cf(\gamma_0, t)  = \Gamma(\sph^1, t),$$ where $\Gamma$ is the solution to the curvature flow problem starting from $\gamma_0$, provided that a unique solution exists.

%If we direct $\gamma_1$ and $\gamma_2$, the direction-based Fr\'echet distance on $\mathcal{J}$ is defined by $$dist_{DF}(\gamma_1, \gamma_2) = \inf \limits_{\varphi} \sup \limits_{x} ||\varphi(x) - x||,$$ where $\varphi: \gamma_1 \to \gamma_2$ is a homeomorphism. 

\subsection{Level-set flow}
\noindent
Given an oriented curve $\gamma \in \mathfrak{J}$, let $\alpha_k \in \mathfrak{J}$ be a sequence of curves approaching $\gamma$ from one side of $\gamma$, and  $\beta_k \in \mathfrak{J}$ be a sequence of curves approaching $\gamma$ from the other  side of $\gamma$. Now, let $\{A_k\}$ be the sequence of nested annuli between $\alpha_k$
and $\beta_k$ and let $A_k(t)$ be the annulus between $\alpha_k(t)$ and $\beta_k(t)$, the time $t$ evolutions of $\alpha_k$ and $\beta_k$ by curvature flow.  The avoidance principle says that if two curves are initially disjoint then they remain disjoint throughout their evolution, for as long as a solution exists. Since $\gamma \subset A_k$ for all $k \in \{1,  ..., \infty\}$, the avoidance principle implies that for any $t>0$, $\cf(\gamma, t) \subset \bigcap \limits_{k=1}^{\infty} A_k(t)$. By  Lauer's Proposition 8.3 \cite{Lauer}, $\bigcap \limits_{k=1}^{\infty} A_k(t)$ is the level-set flow of $\gamma$. Lauer also showed that if the Lebesgue measure of the  initial curve is zero, then the level-set flow immediately becomes a smooth Jordan curve evolving by curvature flow \cite[Theorem 1.2]{Lauer}. %up to some time $T >0$, after which the level-set flow vanishes. 
Therefore, the level-set flow is the unique solution to the curvature flow problem in this case, which implies $\cf(\gamma, t) = \bigcap \limits_{k=1}^{\infty} A_k(t)$. 

\section{Previous results}
\noindent
In this section, we list out the lemmas about Jordan curves  bisecting the $2$-sphere, proved by Michael Dobbins \cite{Dobbins}. These lemmas also apply to Jordan curves in $\R^2$.

\begin{lemma}\label{J phi map}
For each  $\gamma \in \mathfrak{J}$ and $\epsilon > 0$, there is $\delta = \delta(\gamma, \epsilon) > 0$ such
that for every Jordan curve $\eta \subset (\gamma \oplus \delta)$, there is a continuous map  $\varphi : \eta \to \gamma$ such
that for all $x \in \eta, \ \  ||\varphi(x) - x|| < \epsilon$. 

\end{lemma}
\noindent
Note that $\varphi$ is not necessarily bijective, so this does not provide an upper bound on Fr\'echet
distance.

\begin{lemma}\label{J haus}
     Let $\gamma_k \in \mathfrak{J}_T$ and $t_k  \in [0, T)_{\R} $ for $ k \in  \{1, . . . ,\infty\}$. If $\gamma_k \to \gamma_{\infty}$
in Fr\'echet distance and $t_k \to t_{\infty}$, then $\cf(\gamma_k, t_k) \to \cf(\gamma_{\infty}, t_{\infty})$ in Hausdorff distance.
\end{lemma}

\begin{lemma}\label{p_T 1}
Let $\gamma_1, \gamma_2 \in \mathfrak{J}_S$  be a pair of curves  that intersect at finitely many points, and let $p_T \in (\cf(\gamma_1, T) \cap \cf(\gamma_2, T))$ be a point of intersection at time $T \in (0, S)_{\R}$. Then, there is a continuous trajectory $p: [0, T]_{\R} \to \R^2$ such that for all $t \in [0, T]_{\R}$, we have $p(t) \in (\cf(\gamma_1, T) \cap \cf(\gamma_2, T)) $ and $p(T) = p_T$.  
\end{lemma}

\begin{lemma}\label{p_T 2}
Let $\gamma, \eta \in \mathfrak{J}$ be a pair of curves that intersect at only 2 points, and let $x$ be one of the points of intersection.  By  the time  the two points of intersection merge, there is a unique continuous trajectory $p(\gamma, \eta, x): [0, \infty) \to \R^2$ such that for all $t$ we have $p(\gamma, \eta, x; t) \in (\cf(\gamma, t) \cap \cf(\eta, t))$ and $p(\gamma, \eta, x; 0) = x$. Furthermore, $p(\gamma, \eta, x, t)$ is continuous as a function of $\gamma$ and $\eta$ in Fr\'echet distance and $x$ and $t$. 
\end{lemma}

\noindent
The proofs of these lemmas are analogous to Dobbins's proofs. One difference is that in Lemma \ref{J haus} and \ref{p_T 2}, we need to define a closed circular region $\mathbf{D}$  that contains $\gamma_{\infty}$, because we need to use the compactness of $\mathbf{D}$ to show that some subsequence is convergent. The  arguments we need  for the proofs are as follows:

Suppose the diameter of $\gamma_{\infty}$ is $d$.  Let $\mathbf{D} \in \R^2$ be a closed circular region with the diameter $d+1$ such that $\gamma_{\infty} \subset \mathbf{D}$. Since  $\gamma_{\infty}$ and $ \partial \mathbf{D}$ are disjoint and $ \partial \mathbf{D}$ encloses $\gamma_{\infty}$, by the avoidance principle, the shrinking circle $\cf(\partial \mathbf{D}, t)$ encloses $\cf(\gamma_{\infty}, t)$ for all $t \in [0, T]_{\R}$. Therefore, $\cf(\gamma_{\infty}, t) \subset \mathbf{D}$ for all $t \in [0, T]_{\R}$.

Dobbins proved these two lemmas for Jordan curves on the $2$-sphere by using the compactness of the $2$-sphere. Another difference is that for Lemma \ref{p_T 1}, we need to change geodesic curvature in Dobbins's proof to curvature in the argument that the transversal intersection point of two curves evolving by curvature flow continues to be a point of intersection provided that the intersection remains transversal.   % For the reader's convenience, the proofs of Lemma \ref{J haus}, \ref{p_T 1} and \ref{p_T 2} are in Appendix A. 

\section{Proof of continuity}

\begin{lemma} \label{common arc}
Let $\gamma_1, \gamma_2 \in \mathfrak{J}$  be a pair of distinct curves that have a common arc $\mathcal{A}$, i.e. $\mathcal{A} = \gamma_1 \cap \gamma_2$.
%, and $\gamma_2 \setminus \mathcal{A}$ is in one component of $\R^2 \setminus \gamma_1$. %they have no transversal intersection points, i.e. $\gamma_2 \setminus \mathcal{L}$ is only on one side of $\gamma_1$. 
Then  for any $t> 0$, 
%there is no transversal intersection points of $cf(\gamma_1, t)$ and $cf(\gamma_2, t)$ and 
$\cf(\gamma_1, t)$ and  $\cf(\gamma_2, t)$ are disjoint and if 
$\gamma_1 \setminus \mathcal{A} \subset \mathcal{I}(\gamma_2)$ $(\text{or} \ \gamma_2  \setminus \mathcal{A} \subset \mathcal{I}(\gamma_1))$, then 
$\cf(\gamma_1, t) \subset \mathcal{I}(\cf(\gamma_2, t))$  $( \text{or} \ \cf(\gamma_2, t) \subset \mathcal{I}(\cf(\gamma_1, t)))$,   until one of them shrinks to a point.

%remains in the same component of $\R^2 \setminus \gamma_1$ until one of them vanishes.
\end{lemma}
\begin{proof}

% \begin{figure}[!htbp]%
%     \centering
%     \subfloat{{\includegraphics[width=13cm]{1vvv}}}%
%     \caption{$\gamma_1$ and $\gamma_2$ immediately become disjoint, and $\cf(\gamma_1, t)$ remains in the region enclosed by $\cf(\gamma_2, t)$ up to some finite time $T$. }%
%     \label{2}%
% \end{figure}

\begin{figure}[!htbp]%
    \centering
    \includegraphics[width=13cm]{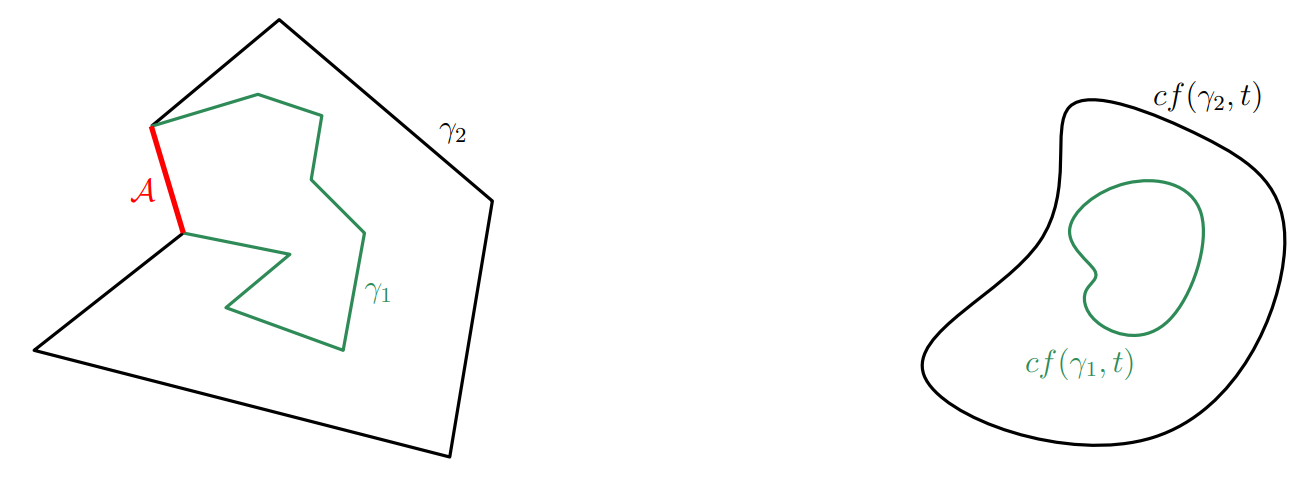}%
    \caption{$\gamma_1$ and $\gamma_2$ immediately become disjoint, and $\cf(\gamma_1, t)$ remains in the region enclosed by $\cf(\gamma_2, t)$ up to some finite time $T$. }%
    \label{2}%
\end{figure}

 Without loss of generality, we assume that $\gamma_1  \subset \mathcal{I}(\gamma_2)$; see Figure \ref{2}. %Sigurd Angenent showed that the number of intersection points immediately becomes finite and then is non-increasing throughout their evolution, provided only that the two curves are initially distinct [1, Theorem 1.3]. 
   % To show that $cf(\gamma_1, t)$ and $cf(\gamma_2, t)$ immediately become disjoint for any $t > 0$, 
   Let $T$ be the first time when one of the curves shrinks to a point as they evolve by curvature flow. 
   First, we want to show that for $t \in (0, T)_{\R}$,  there is no transversal intersection points of $\cf(\gamma_1, t)$ and $\cf(\gamma_2, t)$.    Suppose not. Let $P$ be  a transversal intersection point of  $ \cf(\gamma_1, t_{\infty}) $ and $ \cf(\gamma_2, t_{\infty})$.  Let $\{\alpha_k\}$ be a sequence of Jordan curves such that $\alpha_k \subset \mathcal{I}(\gamma_1)$ and  $\alpha_k \to \gamma_1$ in Fr\'echet distance. Let $\{\beta_k\}$ be a sequence of Jordan curves  such that  $\beta_k \subset \mathcal{O}(\gamma_2)$ and $\beta_k \to \gamma_2$ in Fr\'echet distance.  Then all $\alpha_k$'s are disjoint from all $\beta_k$'s. Let $t_k \to t_{\infty}$ for $t_{\infty} < T$. 
   %$cf(\gamma_2, t_{\infty})$ is not disjoint from $cf(\gamma_1, t_{\infty})$ and let 
    Since $\cf(\gamma_i, t_{\infty})$ for $i = 1, 2$ are smooth, by the implicit function theorem,  there exist  $\varepsilon_1, \varepsilon_2$ such that 
the   component of $(\cf(\gamma_1, t_{\infty}) \oplus \varepsilon_1 )\cap (\cf(\gamma_2, t_{\infty}) \oplus \varepsilon_2) $ with $P$ in it is a region that has a boundary consisting of  4 arcs; %of which the nonadjacent ones are  on  the boundary of either $ \cf(\gamma_1, t_{\infty}) \oplus \varepsilon_1 $ or $  \cf(\gamma_2, t_{\infty}) \oplus \varepsilon_2$;
   % there is an arc of $cf(\gamma_2, t_{\infty})$ going through $P$ from one boundary of $cf(\gamma_1, t_{\infty}) \oplus \varepsilon_1$ to the other, and an arc of $cf(\gamma_1, t_{\infty})$ going through $P$ from one boundary of $cf(\gamma_2, t_{\infty}) \oplus \varepsilon_2$ to the other; 
   see Figure \ref{3}. 
By Lemma \ref{J haus},   $\cf(\alpha_k , t_k) \to \cf(\gamma_1, t_{\infty})$ and $\cf(\beta_k , t_k) \to \cf(\gamma_2, t_{\infty})$ in Hausdorff distance, so  $\cf(\alpha_k, t_{k}) \subset \cf(\gamma_1, t_{\infty}) \oplus \varepsilon_1$ and $\cf(\beta_k, t_{k}) \subset  \cf(\gamma_2, t_{\infty}) \oplus \varepsilon_2$ for $k$ sufficiently large. 
   % Then $cf(\alpha_k, t_k)$ and   $cf(\beta_k, t_k)$ both pass through a component of $(cf(\gamma_1, t_{\infty}) \oplus \varepsilon_1 )\cap (cf(\gamma_2, t_{\infty}) \oplus \varepsilon_2) $ that contains $P$, so 
   Consequently,  it is inevitable for $\cf(\alpha_k, t_k)$ and   $\cf(\beta_k, t_k)$   to intersect with each other, 
 which contradicts the avoidance principle. Therefore, for $t \in (0, T)_{\R}$,  there is no transversal intersection points of $\cf(\gamma_1, t)$ and $\cf(\gamma_2, t)$. 
    \par

Sigurd Angenent showed that the number of intersection points of two different curves immediately becomes finite and then is non-increasing throughout their evolution \cite[Theorem 1.3]{Sigurd}. The intersection is called a tangential intersection point of two curves if the unit tangent vectors are dependent at that point of intersection. If two curves initially have a common arc, then they have infinitely many tangential intersection points. By Angenent's theorem, the number of intersection points  of $\cf(\gamma_1, t)$ and $\cf(\gamma_2, t)$ is finite for $t \in (0, T)_{\R}$, and thus $\cf(\gamma_1, t)$ and $\cf(\gamma_2, t)$ does not share a common arc for  $t \in (0, T)_{\R}$.

 Now we want to show that for $t \in (0, T)_{\R}$,  there is no tangential intersection points of $\cf(\gamma_1, t)$ and $\cf(\gamma_2, t)$.   Suppose not. Then there is a tangential intersection point of $\cf(\gamma_1, t_0)$ and $\cf(\gamma_2, t_0)$ for some $t_0 < T$. By Angenent's theorem, which  says that the set of moments in time $t$ at which the evolution of two curves have  a tangential intersection points is discrete in $(0, T)_{\R}$\cite[Theorem 1.3]{Sigurd}, there exists $\delta > 0$ such that $\cf(\gamma_1, t)$ and $\cf(\gamma_2, t)$ are not tangent when $t \in (t_0 -\delta, t_0)_{\R} \cup (t_0, t_0+\delta)_{\R}$. Then $\cf(\gamma_1, t)$ and $\cf(\gamma_2, t)$     would either have transversal intersection points or be disjoint 
 for $t \in (t_0 -\delta, t_0)_{\R} \cup (t_0, t_0+\delta)_{\R}$. 
  Since we proved that   it is impossible for them to have transversal intersection points for $t \in (0, T)_{\R}$,  $\cf(\gamma_1, t)$ and $\cf(\gamma_2, t)$ are disjoint for  $t \in (t_0 -\delta, t_0)_{\R} \cup (t_0, t_0+\delta)_{\R}$. However, there is a contradiction because  by the avoidance principle, if $\cf(\gamma_1, t)$ and $\cf(\gamma_2, t)$ are disjoint when $t < t_0$, then they remain disjoint when $t = t_0$. Therefore, $\cf(\gamma_1, t)$ and $\cf(\gamma_2, t)$ are disjoint for $t \in (0, T)_{\R}$.

\begin{figure}%
    \centering
    \includegraphics[width=9cm]{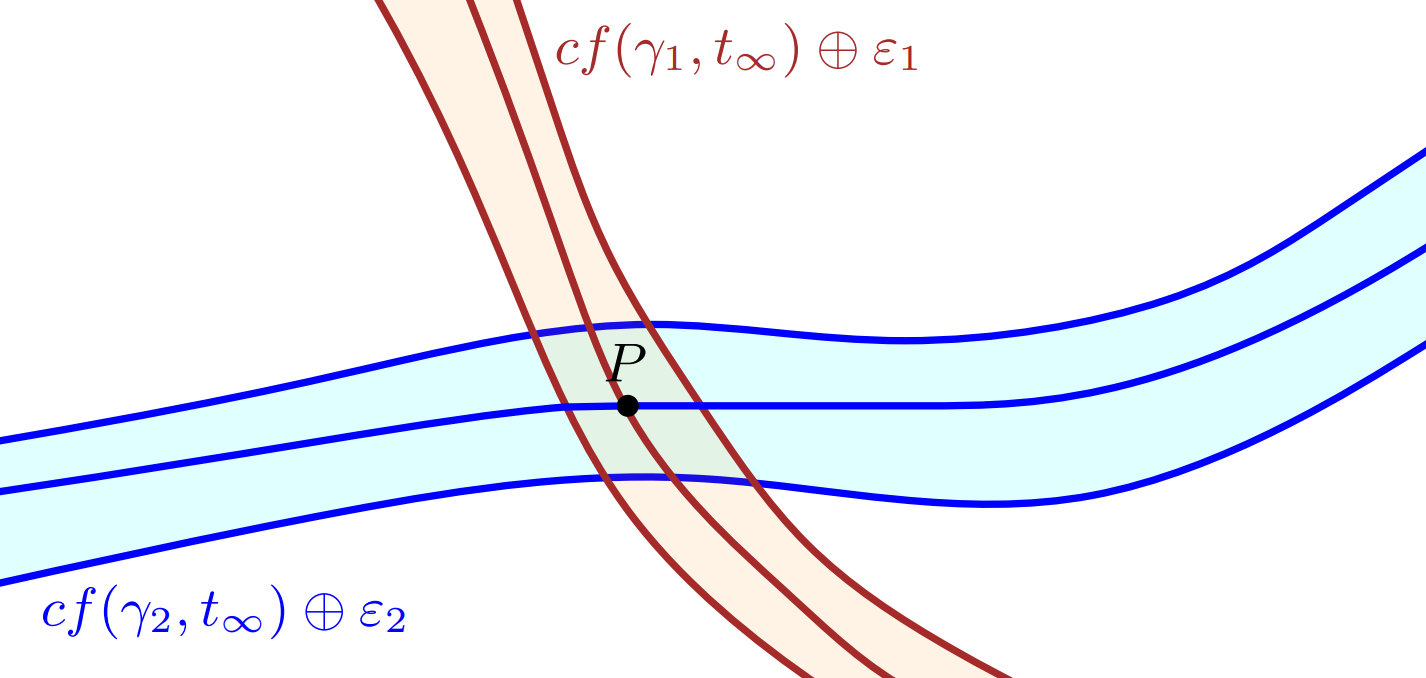}%
    \caption{Since $\cf(\gamma_1, t_{\infty})$ and $\cf(\gamma_2, t_{\infty})$ intersect transversely, there exist $\epsilon_1$ and $\epsilon_2$ such that  $(\cf(\gamma_1, t_{\infty}) \oplus \varepsilon_1) \cap (\cf(\gamma_2, t_{\infty}) \oplus \varepsilon_2)$ is a region that contains $P$ and has a boundary consisting of $4$ arcs.  }%
    \label{3}%
\end{figure}

 % $t \in (t_0, -\delta, t_0 + \delta)$. Then $cf(\gamma_1, t_0)$ and $cf(\gamma_2, t_0)$ having a tangential intersection point contradicts the avoidance principle, because by the avoidance principle, if $cf(\gamma_1, t)$ and $cf(\gamma_2, t)$ are disjoint for $t < t_0$, then they will be disjoint when $t = t_0$. Therefore,  there is no tangential intersection points of $cf(\gamma_1, t)$ and $cf(\gamma_2, t)$, for $t \in (0, T)$. 
    \par

    Now we want to show that $\cf(\gamma_1, t) \subset \mathcal{I}(\cf(\gamma_2, t))$ for $t \in (0, T)_{\R}$. Suppose not. Then, $\cf(\gamma_1, t) \subset \mathcal{O}(\cf(\gamma_2, t))$ for $t \in (0, T)_{\R}$. Let 
    $\{\zeta_k \}$ be a sequence of smooth Jordan curves  such that  $ \zeta_k \subset \mathcal{I}(\gamma_1) $ and $\zeta_k \to \gamma_1$ in Fr\'echet distance. 
    Then  $\zeta_k \subset \mathcal{I}(\gamma_2)$ for all $k \in \{1, ..., \infty\}$.  
    Let $\varphi_k: \mathbf{S}^1 \to \zeta_k$ and $\psi: \mathbf{S}^1 \to \gamma_1$  be two parameterizations such that $\sup_{x \in \mathbf{S}^1} || \varphi_k(x) - \psi(x) || \to 0$.     
If we fix a point $x_0 \in \mathbf{S}^1$, then    $|| \varphi_k(x_0) - \psi(x_0)||    \to 0$, i.e. $\varphi_k(x_0)$ converges to $\psi(x_0) \in \gamma_1$.    
%Let $t_k \to t_{\infty}$.  
Suppose $\zeta_k \in \mathfrak{J}_{T_k}$. 
Define $\mathrm{g}_k: \varphi_k(x_0) \times [0, T_k) \to \R^2$ by   the trajectory of $\varphi_k(x_0)$ as $\zeta_k$ evolves by curvature flow. %Then $\mathrm{g}_k(\varphi_k(x_0), 0) = \varphi_k(x_0)$. %denote the trajectory of $\varphi_k(x_0)$ as $\zeta_k$ evolves by curvature flow.  
By Lemma \ref{J haus}, $\cf(\zeta_k, t_{k}) \to \cf(\gamma_1, t_{\infty})$ in Hausdorff distance, so 
for $k$ sufficiently large, $\mathrm{g}_k(\varphi_k(x_0), t_k) \in \cf(\gamma_1, t_{\infty}) \oplus \varepsilon$ for any $\varepsilon >0$.  Since we assumed that  $\cf(\gamma_1, t_{\infty}) \subset \mathcal{O}(\cf(\gamma_2, t_{\infty}))$,  we have $\mathrm{g}_k(\varphi_k(x_0), t_k) \in \mathcal{O}(\cf(\gamma_2, t_{\infty}))$, which implies that 
 $\cf(\zeta_k, t_k) $ crosses  $\cf(\gamma_2, t_k)$ at some point as $t_k \to t_{\infty}$. However, since  each $\zeta_k$'s was initially disjoint from  $\gamma_2$,  their intersection contradicts  the avoidance principle.
 %they should remain disjoint throughout their evolution by curvature flow. Therefore, there is a contradiction. 
%which contradicts the avoidance principle since   all $\zeta_k$'s were initially lying in  $\mathcal{I}(\gamma_1)$,    $(\varphi_k(x_0), t_{\infty})$ lying in $B(cf(\gamma_1, t_{\infty}))$ indicates that 
 %$cf(\zeta_k, t_k) $ crosses  $cf(\gamma_1, t_k)$ at some point as $t_k \to t_{\infty}$. 
 %, which contradicts the avoidance principle.  
\end{proof}
 % \text{  }

\begin{lemma}\label{Existence}
Let $\gamma \in \mathfrak{J}_T$, and let $x \in \gamma$ be a point. Then, there is a curve $\eta \in \mathfrak{J}$ such that it does not shrink to a point before  $\gamma$ does, only intersects $\gamma$ at two points, which are $x$ and some other point, and  the number of intersection points of $\cf(\gamma, t)$ and $\cf(\eta, t)$ remains 2  for any $t \in (0, T)_{\R}$. Furthermore, $\eta$ is continuous in Fr\'echet distance as a function of $x$. 
\end{lemma}

\begin{proof}
 Suppose the diameter of $\gamma$ is $2r$, for some $r > 0$.  Let $\mathbf{D} \in \R^2$ be a closed circular region with  diameter $4r$ such that $\gamma  \in \mathbf{D}$.  Then we can find another closed circular region  $\mathbf{E}$ with  diameter  $ d > 4\sqrt{2T} + 4r $ such that $\mathbf{E}$ and $\mathbf{D}$ have the same center. Let $\mathbf{\Gamma}$ be the closed region enclosed by $\gamma$.  Let $A$  be the closed annulus between two concentric circles with  radii $1$ and $2$ in $\R^2$. Let  $\varphi: A \to \overline{\mathbf{D} \setminus \mathbf{\Gamma}}$ be a  homeomorphism.  For each $x \in \gamma$, we have points $\varphi^{-1}(x), -\varphi^{-1}(x) , 2 \varphi^{-1}(x), -2\varphi^{-1}(x) \in \partial A$; see Figure \ref{4}. 
 Let $\mathbf{U}$ and $\mathbf{V}$ denote two arcs of $\gamma$ divided by $x$ and $\varphi(-\varphi^{-1}(x))$. Then there is  a homotopy $\sigma(\cdot, s) = sf +(1-s)g$ between the path $f: [0,1] \to \mathbf{U}$ and the path $g: [0,1] \to \mathbf{V}$ such that $\sigma(\cdot, 0) = g$ and $\sigma(\cdot, 1) = f$.  Given a point $x \in \gamma$, we denote each path generated by the  homotopy $\sigma$ by $\eta_1(x, s)$ for some $s \in [0, 1]$. Let 
$\eta_2(x) = \varphi([\varphi^{-1}(x), 2\varphi^{-1}(x)]_{\R^2})$, and 
$\eta_3(x) = \varphi([-\varphi^{-1}(x), -2\varphi^{-1}(x)]_{\R^2})$.
%and $\eta_4(x)$ be the image of the arc on $\partial A$ going from $2\varphi^{-1}(x)$ to $-2\varphi^{-1}(x)$ counterclockwise under $\varphi$.  

Given two distinct points, $v_1, v_2 \in \partial \mathbf{D}$, let $\pi_1: \partial \mathbf{D} \times \partial \mathbf{D} \to \partial \mathbf{E}$ be a map defined by $\pi_1(v_1, v_2) = w$ such that $v_1, v_2, w$ are collinear, and $|| w -  v_1|| < ||w - v_2||$. Likewise, let $\pi_2: \partial \mathbf{D} \times \partial \mathbf{D} \to \partial \mathbf{E}$ be a map defined by $\{v_1, v_2\} \mapsto w$ such that $v_1, v_2, w$ are collinear, and $|| w -  v_1|| >  ||w - v_2||$. Let $\eta_4(x) = [\varphi(2\varphi^{-1}(x)),  \pi_1(\varphi(2\varphi^{-1}(x)),  \varphi(-2\varphi^{-1}(x)))]_{\R^2}$,  $\eta_5(x) =  [\varphi(-2\varphi^{-1}(x)),  \pi_2( \varphi(  2\varphi^{-1}(x)), $ $ \varphi(-2\varphi^{-1}(x)))]_{\R^2}$, and  $\eta_6(x) \subset \partial \mathbf{E}$ be an arc of  $\partial \mathbf{E}$ going from  $\pi_1(\varphi(2  \varphi^{-1}  (x)),  \varphi(-2\varphi^{-1}(x)))$ to  $\pi_2(\varphi(2\varphi^{-1}(x)),  \varphi(-2\varphi^{-1}(x)))$ clockwise.

Now we can construct a Jordan curve, $\eta(x, s)$, consisting of $\eta_1(x, s),  \eta_2(x), \eta_3(x),  \eta_4(x), \eta_5(x)$ and $ \eta_6(x)$. Note that $\eta(x, s)$ only intersects $\gamma$ at two points, $x$ and $\varphi(-\varphi^{-1}(x))$.  Let the closed region enclosed by $\eta(x)$ be $\mathbf{H}$. 
Note that  a circle with initial radius $r_0$ can be seen to evolve by curvature flow to a circle with radius $r(t) = \sqrt{r_0^2 -2t}$, and thus the circle disappears at $t = r_0^2/2$. 
Since $d  > 4\sqrt{2T} + 4r$, the difference between the radii of $\mathbf{E}$ and $\mathbf{D}$ is  greater than $2\sqrt{2T}$. Therefore, a circle with radius $\sqrt{2T}$ can be fitted in the interior of  $\mathbf{H} \setminus \mathbf{D}$, and its lifetime is $T$. By avoidance principle, the lifetime of $\eta(x)$ is greater than $T$.

\begin{figure}
\centering
\includegraphics[width = 13.2cm]{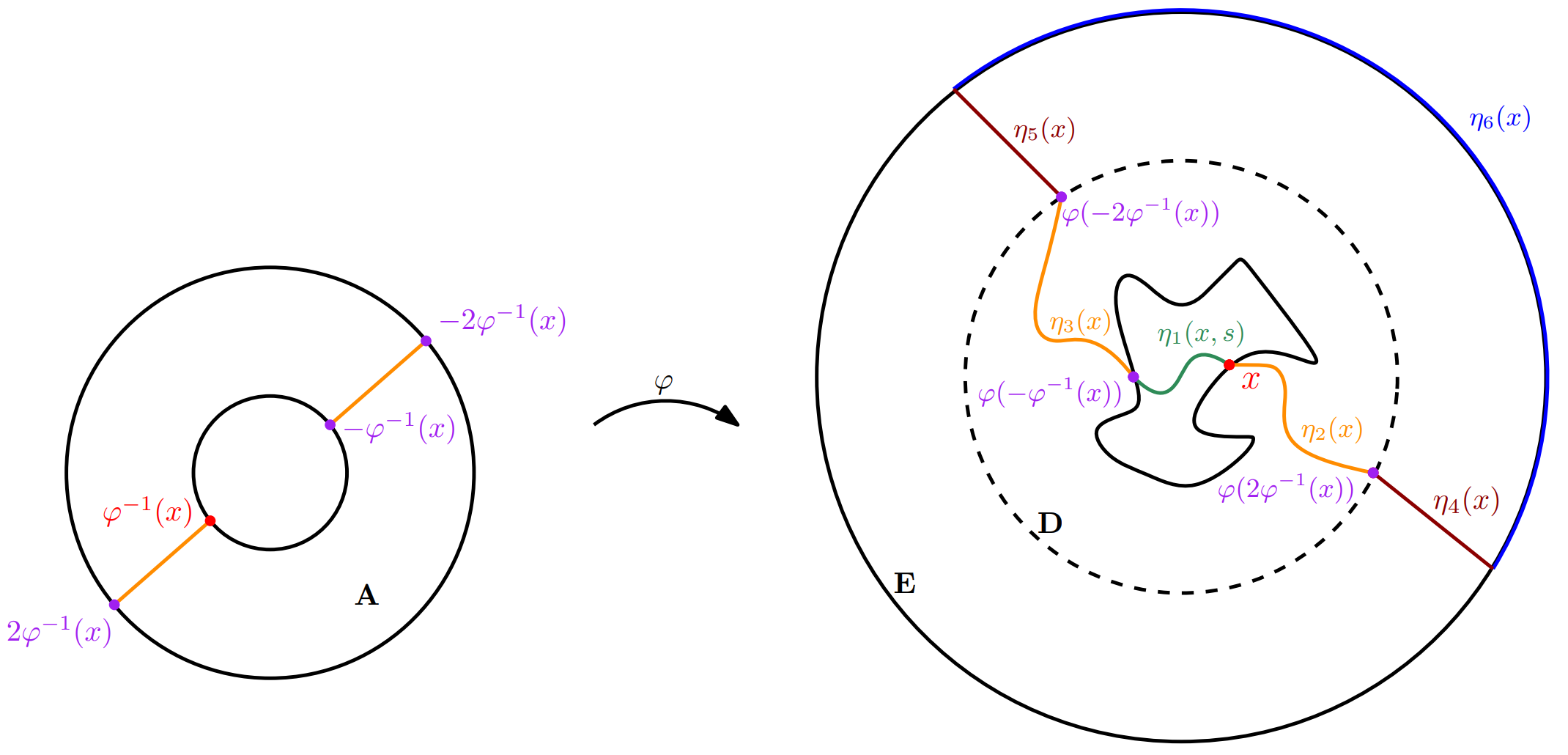}
\caption{We use a homeomorphism $\varphi$ and a homotopy $\sigma$ to define a Jordan curve,  $\eta(x, s)$  consisting of $6$ arcs. }
\label{4}
\end{figure}

Now we want to show that $\eta $ is continuous. Since $\varphi$ is a homeomorphism, as $x_k \to x_{\infty}$, we have $\varphi^{-1}(x_k) \to \varphi^{-1}(x_{\infty})$, $-\varphi^{-1}(x_k) \to -\varphi^{-1}(x_{\infty})$, $2\varphi^{-1}(x_k) \to 2\varphi^{-1}(x_{\infty})$, and $-2\varphi^{-1}(x_k) \to -2\varphi^{-1}(x_{\infty})$, so $[\varphi^{-1}(x_k), 2\varphi^{-1}(x_k)]_{\R^2} \to [\varphi^{-1}(x_{\infty}), 2\varphi^{-1}  (x_{\infty})]_{\R^2}$ and $[-\varphi^{-1}(x_k), -2\varphi^{-1}(x_k)]_{\R^2} \to  [-\varphi^{-1}(x_{\infty}), -2\varphi^{-1}(x_{\infty})]_{\R^2}$
in Fr\'echet distance. By the Heine-Cantor theorem, $\varphi$ is uniformly continuous, so $\eta_i(x_k) \to \eta_i(x_{\infty})$  in Fr\'echet distance for $i \in \{2, 3\}$. Since $\varphi(2\varphi^{-1}(x_k)) \to  \varphi(2\varphi^{-1}(x_{\infty}))   $, $\varphi(-2\varphi^{-1}(x_k)) \to  \varphi(-2\varphi^{-1} (x_{\infty}))$ and how $\pi_1$ and $\pi_2$ are defined, we have 
$\eta_i(x_k) \to \eta_i(x_{\infty})$ in Fr\'echet distance, for $i \in \{4, 5, 6\}$.

 Note that  every pair of $\eta(x, 0)$, $\eta(x, 1)$ and $\gamma$  share a common arc. By Lemma \ref{common arc}, $\cf(\eta(x, 0), t)$  $\cf(\eta(x, 1), t)$ and $\cf(\gamma, t)$ are disjoint from each other for $t \in (0, T)_{\R}$. Without the loss of generality, we assume that $\eta(x, 0) \setminus (\eta(x, 0) \cap \eta(x, 1)) \subset \mathcal{I}(\eta(x, 1))$. Then 
$ \gamma  \setminus (\gamma \cap \eta(x, 0) ) \subset \mathcal{O}(\eta(x, 0))$, and $ \gamma  \setminus (\gamma \cap \eta(x, 1) ) \subset \mathcal{I}(\eta(x, 1))$. By Lemma \ref{common arc}, both $\cf(\gamma, t)$ and $\cf(\eta(x, 0), t)$ are lying in $\mathcal{I}(\cf(\eta(x, 1),  t))$, with $\cf(\gamma, t) \subset \mathcal{O}(\cf(\eta(x, 0), t))$. 
 For each $t_{\infty} \in (0, T)_{\R}$, the family of Jordan curves $ \mathcal{F} = \{\eta(x, s) \ | \ 0 \leq s \leq 1\}$ can be divided into three classes, $\mathcal{C}_1$, $\mathcal{C}_2$, and 
 $\mathcal{C}_3$. Class $\mathcal{C}_1$ contains curves  $\eta(x, s) \in \mathcal{F}$ such that $\cf(\gamma, t_{\infty}) \subset \mathcal{O}(\cf(\eta(x, s), t_{\infty}))$. 
 Class $\mathcal{C}_2$ contains curves  $\eta(x, s) \in \mathcal{F}$ such that $\cf(\gamma, t_{\infty}) \cap \cf(\eta(x, s), t_{\infty}) \neq \emptyset$. Class $\mathcal{C}_3$ contains curves  $\eta(x, s) \in \mathcal{F}$ such that $\cf(\gamma, t_{\infty}) \subset \mathcal{I}(\cf(\eta(x, s), t_{\infty}))$. Since $\eta(x, 0) \in \mathcal{C}_1$ and $\eta(x, 1) \in \mathcal{C}_3$, $\mathcal{C}_1$ and $\mathcal{C}_3$ are both non-empty.  Let $R$ be the radius of the largest  circle that can be inscribed within $\cf(\gamma, t_{\infty})$. By the definition of $\mathcal{C}_1$ and $\mathcal{C}_3$, given two curves $\mu \in \mathcal{C}_1$ and $\mu' \in \mathcal{C}_3$, $\cf(\mu, t_{\infty})$ and $\cf(\mu', t_{\infty})$ are  at least $R$  apart; see Figure \ref{5}. We want to show that class $\mathcal{C}_2$ is also non-empty. Suppose not. 
Let $s_0 = \inf\{ s \ | \ \eta(x, s) \in \mathcal{C}_3, \ 0 \leq s \leq 1 \}$.    %Since class $\mathcal{A}$ is not empty, $s_0 > 0$. 
%Then we claim that $\cf(\eta(x, s_k), t_{\infty})$ and $\cf(\eta(x, s_0), t_{\infty})$ 

\textbf{Case I.} Suppose $\eta(x, s_0) \in \mathcal{C}_3 $. There is a sequence of curves $\eta(x, s_k) \to \eta(x, s_0)$ in Fr\'echet distance as $s_k \to s_0$ from below such that $\eta(x, s_k)
 \in \mathcal{C}_1$.  Then, $\cf(\eta(x, s_k), t_{\infty})$ and $\cf(\eta(x, s_0), t_{\infty})$  are at least $R$ apart. 
 Since, by Lemma \ref{J haus}, $\cf(\eta(x, s_k), t_{\infty}) \to \cf(\eta(x, s_0), t_{\infty})$ in Hausdorff distance, $\cf(\eta(x, s_k), t_{\infty})$ and $\cf(\eta(x, s_0), t_{\infty})$ are within a Hausdorff distance of $\varepsilon$, for any $\varepsilon > 0$ and all $k$ sufficiently large. For $\varepsilon < R$, we have $\cf(\eta(x, s_k), t_{\infty})$ and $\cf(\eta(x, s_0), t_{\infty})$    less than $R$ apart, and thus there is a contradiction.

 % Since $\cf(\eta(x, s_k), t_{\infty})$ and $\cf(\eta(x, s_k), t_{\infty})$ are two compact, disjoint sets, they are bounded away by some distance.
 % Since, by Lemma \ref{J haus}, $\cf(\eta(x, s_k), t_{\infty}) \to \cf(\eta(x, s_0), t_{\infty})$ in Hausdorff distance, $\cf(\eta(x, s_k), t_{\infty})$ and $\cf(\eta(x, s_0), t_{\infty})$ are infinitely close to each other for sufficiently large $k$. Therefore, $\cf(\eta(x, s_k), t_{\infty})$ and $\cf(\eta(x, s_k), t_{\infty})$ are not bounded away by the diameter of $\cf(\gamma, t_{\infty})$ for sufficiently large $k$, implying that $\cf(\gamma, t_{\infty}) \subset \mathcal{O}(\cf(\eta(x, s_0), t_{\infty}))$. Then 
 % $\eta(x, s_0) \in \mathcal{C}_1 $, and thus there is a contradiction.

%which contradicts that $cf(\ell(w, s_k), t_{\infty}) \in \mathcal{A}$. 

\textbf{Case II.}
Suppose $\eta(x, s_0) \in \mathcal{C}_1 $. 
There is a sequence of curves $\zeta(s_k) \to \zeta(s_0)$ in Fr\'echet distance as $s_k \to s_0$ from above such that $\eta(x, s_k
) \in \mathcal{C}_3$. 
Then, $\cf(\eta(x, s_k), t_{\infty})$ and $\cf(\eta(x, s_0), t_{\infty})$  are  at least $R$ apart. 
 Since, by Lemma \ref{J haus}, $\cf(\eta(x, s_k), t_{\infty}) \to \cf(\eta(x, s_0), t_{\infty})$ in Hausdorff distance, $\cf(\eta(x, s_k), t_{\infty})$ and $\cf(\eta(x, s_0), t_{\infty})$ are within a Hausdorff distance of $\varepsilon$, for any $\varepsilon > 0$ and all $k$ sufficiently large. For $\varepsilon < R$, we have $\cf(\eta(x, s_k), t_{\infty})$ and $\cf(\eta(x, s_0), t_{\infty})$    less than $R$ apart, and thus there is a contradiction.

% Since $\cf(\eta(x, s_k), t_{\infty})$ and $\cf(\eta(x, s_k), t_{\infty})$ are two compact, disjoint sets, they are bounded away by some distance.
%  Since, by Lemma \ref{J haus}, $\cf(\eta(x, s_k), t_{\infty}) \to \cf(\eta(x, s_0), t_{\infty})$ in Hausdorff distance, $\cf(\eta(x, s_k), t_{\infty})$ and $\cf(\eta(x, s_k), t_{\infty})$ are infinitely close to each other for sufficiently large $k$. Therefore, $\cf(\eta(x, s_k), t_{\infty})$ and $\cf(\eta(x, s_k), t_{\infty})$ are not bounded away by the diameter of $\cf(\gamma, t_{\infty})$ for sufficiently large $k$, 
%  implying that $\cf(\gamma, t_{\infty}) \subset \mathcal{I}(\cf(\eta(x, s_0), t_{\infty}))$. Then 
%  $\eta(x, s_0) \in \mathcal{C}_3 $, and thus there is a contradiction.

Therefore, class $\mathcal{C}_2$ is also non-empty. Now let us show that, based on the value of $s$, these three classes correspond to three intervals on $[0, 1]_{\R}$; see Figure \ref{5}. We want to show that  for any $a \in [0, 1]_{\R}$, if $\eta(x, a) \in \mathcal{C}_1$,  then $\eta(x, b) \in \mathcal{C}_1$ for any $0 \leq b <  a$. Let $\mathcal{A}$ be the common arc of $\eta(x, a)$ and $\eta(x, b)$. Since $\eta(x, b) \setminus \mathcal{A} \subset \mathcal{I}(\eta(x, a))$, by Lemma \ref{common arc}, $\cf(\eta(x, b) , t_{\infty}) \subset \mathcal{I}( \eta(x, a) , t_{\infty}))$. Since $\eta(x, a) \in \mathcal{C}_1$, $\cf(\gamma, t_{\infty}) \subset \mathcal{O}(\cf(\eta(x, a), t_{\infty}))$. Therefore, $\cf(\gamma, t_{\infty}) \subset \mathcal{O} (\cf(\eta(x, b), t_{\infty}))$, and thus $\eta(x, b) \in \mathcal{C}_1$. Similarly, we can show that 
for any $a \in [0, 1]_{\R}$, if $\eta(x, a) \in \mathcal{C}_3$,  then $\eta(x, b) \in \mathcal{C}_3$ for any $a < b \leq 1$.

Let $\nu_1(x) = \inf\{ s \ | \ \eta(x, s) \in \mathcal{C}_2, 0 \leq s \leq 1\}$, and 
$\nu_2(x) = \sup\{ s \ | \ \eta(x, s) \in \mathcal{C}_2, 0 \leq s \leq 1\}$. We want to show that $\eta(x, \nu_1(x)) \in \mathcal{C}_2$ for any $x \in \gamma$. Suppose not.   Then $  \eta(x, \nu_1(x)) \in \mathcal{C}_1$, and there exists $\varepsilon_1 > 0$ such that $(\cf(\eta(x, \nu_1(x)), t_{\infty})  \oplus \varepsilon_1 ) \cap \cf(\gamma, t_{\infty}) = \emptyset$, for $t_{\infty} \in (0, T)_{\R}$. There is a sequence of curves $\eta(x, s_k) \to \eta(x, \nu_1(x))$ in Fr\'echet distance as $s_k \to \nu_1(x)$ from above such that $\eta(x, s_k) \in \mathcal{C}_2$. Let $t_k \to t_{\infty}$.  By Lemma \ref{J haus}, 
 $\cf(\eta(x, s_k), t_k) \to \cf(\eta(x, \nu_1(x)), t_{\infty})$ in Hausdorff distance, implying that 
$\cf(\eta(x, s_k), t_k) \subset \cf(\eta(x,  \nu_1(x)),  t_{\infty})  \oplus \varepsilon_1 $ for $k$ sufficiently large, and consequently
$\cf(\eta(x, s_k), t_{\infty}) \cap \cf(\gamma, t_{\infty}) = \emptyset$ for $k$ sufficiently large. This 
 contradicts that $\eta(x, s_k) \in \mathcal{C}_2$.

 Furthermore, $\cf(\eta(x, \nu_1(x)), t_{\infty}) $ only intersects $\cf(\gamma, t_{\infty})$ at one point. Suppose not. Then $\cf(\eta(x, \nu_1(x)), t_{\infty}) $ intersects $\cf(\gamma, t_{\infty})$ at two points. 
 There exists $\varepsilon_2 > 0$ such that the boundary of $(\cf(\eta(x, \nu_1(x)), t_{\infty})  \oplus \varepsilon_2 ) $ consists of two Jordan curves that both intersect $\cf(\gamma, t_{\infty})$ at two points. There is a sequence of curves $\eta(x, s_k) \to \eta(x, \nu_1(x))$ in Fr\'echet distance as $s_k \to \nu_1(x)$ from below such that $\eta(x, s_k) \in \mathcal{C}_1$. Let $t_k \to t_{\infty}$.
 By Lemma \ref{J haus}, 
 $\cf(\eta(x, s_k), t_k) \to \cf(\eta(x, \nu_1(x)), t_{\infty})$ in Hausdorff distance, implying that 
$\cf(\eta(x, s_k), t_k) \subset \cf(\eta(x,  \nu_1(x)),  t_{\infty})  \oplus \varepsilon_2 $ for $k$ sufficiently large, and consequently
$\cf(\eta(x, s_k), t_{\infty}) \cap \cf(\gamma, t_{\infty}) \neq \emptyset$ for $k$ sufficiently large. This 
 contradicts that $\eta(x, s_k) \in \mathcal{C}_1$.
 We can use analogous arguments to show that  $\eta(x, \nu_2(x)) \in \mathcal{C}_2$ and $\cf(\eta(x, \nu_2(x)), t_{\infty})$ only intersects $\cf(\gamma, t_{\infty})$ at one point.

 Let $\nu_0(x) = (\nu_1(x)+ \nu_2(x))/2$. %\eta(x, (\nu_1(x)+ \nu_2(x))/2)$.  
 Note that $\cf(\eta(x, \nu_0(x)), t_{\infty})  \subset  \mathcal{O}(\cf(\eta(x, \nu_1(x)),  
 t_{\infty}))$ and $\cf(\eta(x, \nu_0(x)), t_{\infty}) \subset  \mathcal{I}(\cf(\eta(x,   \nu_2(x)), t_{\infty}))$. 
  Since $\cf(\eta(x, \nu_1(x)), t_{\infty})$ and $\cf(\eta(x,     \nu_2(x)), t_{\infty})$ are both tangent to $\cf(\gamma, t_{\infty})$ with $\cf(\eta(x, \nu_1(x)), t_{\infty}) \subset \mathcal{I}(\cf(\eta(x,   \nu_2(x)),   t_{\infty}))$, 
 $\cf(\eta(x, \nu_0(x)), t_{\infty}) $ intersects $  \cf(\gamma, t_{\infty}) $ at at least $2$ points.  Since Angenent proved that the point of intersection does not increase throughout the evolution by curvature flow \cite{Sigurd},  and $\eta(x, \nu_0(x))$ initially intersects $\gamma$ at two points, $\cf(\eta(x, \nu_0(x)), t_{\infty})$ must intersect $\cf(\gamma, t_{\infty})$ at two points. 
 %the number of intersection points of $cf(\eta(x, (\nu_1(x) + \nu_2(x))/2), t_{\infty}) $ and $ cf(\gamma, t_{\infty})$ could be $2, 1$ or $0$. Since $\eta(x, \nu_i(x)) \in \mathcal{C}_2$, $cf(\eta(x, \nu_i(x)), t_{\infty}) \cap cf(\gamma, t_{\infty} ) \neq \emptyset$, for $i = 1, 2$. Therefore, we can find two arcs of $cf(\gamma, t_{\infty})$ from$cf(\eta(x, \nu_1(x)), t_{\infty})$ to $cf(\eta(x, \nu_2(x)), t_{\infty})$. Since $cf(\eta(x, (\nu_1(x) + \nu_2(x))/2), t_{\infty}) \subset  \mathcal{O}(cf(\eta(x, \nu_1(x)), t_{\infty})) $ and $cf(\eta(x, (\nu_1(x) + \nu_2(x))/2), t_{\infty}) \subset  \mathcal{I}
%Suppose not. Then $  \eta(x, \nu_2(x)) \in \mathcal{C}_3$, and there exists $\varepsilon_2 > 0$ such that $(cf(\eta(x, \nu_1(x)), t_{\infty})  \oplus \varepsilon_2 ) \cap cf(\gamma, t_{\infty}) = \emptyset$. There is a sequence of curves $\eta(x, s_k) \to \eta(x, \nu_2(x))$ in Fr\'echet distance as $s_k \to s_0$ from below such that $\eta(x, s_k) \in \mathcal{C}_2$. By Lemma \ref{J haus}, $cf(\eta(x, s_k), t_{\infty}) \to cf(\eta(x, \nu_2(x)), t_{\infty})$ in Hausdorff distance, implying that $cf(\eta(x, s_k), t_{\infty}) \subset cf(\eta(x, \nu_2(x)),  t_{\infty})  \oplus \varepsilon_2 $ for $k$ sufficiently large, and consequently $cf(\eta(x, s_k), t_{\infty}) \cap cf(\gamma, t_{\infty}) = \emptyset$ for $k$ sufficiently large. This  contradicts that $\eta(x, s_k) \in \mathcal{C}_2$. Therefore,  $\eta(x, \nu_2(x)) \in \mathcal{C}_2$ for any $x \in \gamma$.

\begin{figure}
\centering
\includegraphics[width = 13cm]{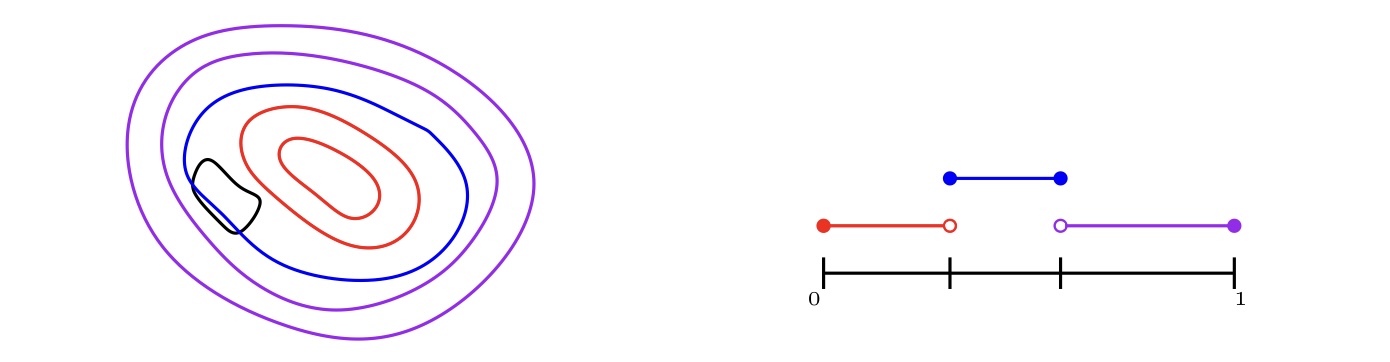}
\caption{On the left, the black curve represents  $\cf(\gamma, t)$, and  the red, blue, purple curves are  curves in $\mathcal{C}_1, \mathcal{C}_2, \mathcal{C}_3$, evolving by curvature flow  at time $t \in (0, T)_{\R}$. On the right are three intervals corresponding to  $\mathcal{C}_1, \mathcal{C}_2, \mathcal{C}_3$. The interval corresponding to $\mathcal{C}_2$  is closed.  }
\label{5}
\end{figure}

Now we want to show that $\nu_1$ is continuous. Suppose not. Let $x_{\infty} \in \gamma$.  Then there exists a sequence $x_k \to x_{\infty}$ such that 
 $\nu_1(x_k)$ is bounded away $  \nu_1(x_{\infty}) $ by some $ \varepsilon > 0$. 
By the compactness of the unit interval, we can restrict $\{\nu_1(x_k)\}$ to a subsequence $\{\nu_1(x_i)\}$ that converges to $p$.

\par

\textbf{Case I.} 
Assume that $p < \nu_1(x_{\infty})$. Then there is $s_0 \in [0, 1]_{\R}$ such that $p < s_0 < \nu_1(x_{\infty})$ and $\eta(x_i, s_0) \in \mathcal{C}_2$ for $i$ sufficiently large.  Let $t_{\infty} \in (0, T)_{\R}$. 
Since $s_0 < \nu_1(x_{\infty})$, $\eta(x_{\infty}, s_0) \in \mathcal{C}_1$, and then $\cf(\eta(x_{\infty}, s_0), t_{\infty}) \cap \cf(\gamma, t_{\infty}) = \emptyset$. Therefore, there exists $\delta_1 >0$ such that $(\cf(\eta(x_{\infty}, s_0), t_{\infty})  \oplus \delta_1 ) \cap \cf(\gamma, t_{\infty}) = \emptyset$. Let $t_i \to t_{\infty}$. 
Since $x_i \to x_{\infty}$, $\eta(x_i, s_0) \to \eta(x_{\infty}, s_0)$ in Fr\'echet distance. By Lemma \ref{J haus}, $\cf(\eta(x_i, s_0), t_i) \to \cf(\eta(x_{\infty}, s_0), t_{\infty})$ in Hausdorff distance, implying that 
$\cf(\eta(x_i, s_0), t_{i}) \subset \cf(\eta(x_{\infty}, s_0),  t_{\infty})  \oplus \delta_1 $ for $i$ sufficiently large, and consequently 
$\cf(\eta(x_i, s_0), t_{\infty}) \cap \cf(\gamma, t_{\infty}) = \emptyset$.
This 
 contradicts that $\eta(x_i, s_0) \in \mathcal{C}_2$. %Therefore, $\nu_1$ is continuous. 

\textbf{Case II} 
Assume that $p > \nu_1(x_{\infty})$. Then there is $s_1 \in [0, 1]_{\R}$ such that $\nu_1(x_{\infty}) < s_1 < \nu_1(x_{i})$ and $\cf(\eta(x_{\infty}, s_1), t_{\infty}) $
intersects $ \cf(\gamma, t_{\infty})$ at two points. 
Therefore, there exists $\delta_2 >0$ such that the boundary of $\cf(\eta(x_{\infty}, s_0), t_{\infty})  \oplus \delta_2  $ consists of two Jordan curves that both intersect $\cf(\gamma, t_{\infty})$ at two points. 
Since $\nu_1(x_i ) \to p$ and 
$s_1 < p$, $\eta(x_i, s_1) \in \mathcal{C}_1$ for $i$ sufficiently large. 
Let $t_i \to t_{\infty}$. Since $x_i \to x_{\infty}$, $\eta(x_i, s_1) \to \eta(x_{\infty}, s_1)$ in Fr\'echet distance. By Lemma \ref{J haus}, $\cf(\eta(x_i, s_1), t_i) \to \cf(\eta(x_{\infty}, s_1), t_{\infty})$ in Hausdorff distance, implying that 
$\cf(\eta(x_i, s_1), t_{i}) \subset \cf(\eta(x_{\infty}, s_1),  t_{\infty})  \oplus \delta_2 $ for $i$ sufficiently large, and consequently 
$\cf(\eta(x_i, s_1), t_{\infty}) $ intersects $ \cf(\gamma, t_{\infty})$ at two points.
This 
 contradicts that $\eta(x_i, s_1) \in \mathcal{C}_1$.

 Therefore, $\nu_1$ is continuous. Similarly, we can show that $\nu_2$ is also continuous. For each $x \in \gamma$, since $\nu_1, \nu_2$  are continuous, and $\gamma$ is compact, by Heine-Cantor theorem, $\nu_1$ and $\nu_2$ are uniformly continuous. Therefore, $\eta(x, \nu_0(x))$ is continuous in Fr\'echet distance. Therefore, $\cf(\eta(x, \nu_0(x)), t_{\infty}) $ is the curve that has the desired properties. 
\end{proof}
%\text{}

%Suppose there exists $\varepsilon >0$ such that for all $\delta >0$, $|x_1  - x_2| < \delta$, but $|\nu_1(x_1) - \nu_1(x_2)| > \varepsilon$. 

%For each $x \in \gamma$, there are unique $\nu_1(x)$ and  $\nu_2(x)$.   Therefore, $\nu_i(x_k) \to \nu_i(x_{\infty}$ for $i \in \{1, 2\}$. Therefore, both $\nu_1$ and $\nu_2$ are continuous. Then $ \eta_1(x_k, \dfrac{\nu_1(x_k)+ \nu_2(x_k)}{2}) \to \eta_1(x_{\infty}, \dfrac{\nu_1(x_{\infty})+ \nu_2(x_{\infty})}{2}) $
%in F\'echet distance. Therefore, the curve $\eta(x_k, \dfrac{\nu_1(x_k )+ \nu_2(x_k )}{2}) \break \to \eta(x_{\infty}, \dfrac{\nu_1(x_{\infty})+ \nu_2(x_{\infty})}{2})$ in F\'echet distance. 

%with $a< b$, if $\eta(x, a) \in \mathcal{C}_1$ and $\eta(x, b) \in \mathcal{C}_1$, then $\eta(x, c) \in \mathcal{C}_1$ for any $a < c< b$. Let $\mathcal{A}$ be the common arc of $\eta(x, a), \eta(x, b)$ and $\eta(x, c)$. Since $\eta(x, c) \setminus \mathcal{A} \subset B(\eta(x, b))$ and $\eta(x, a) \setminus \mathcal{A} \subset U(\eta(x, c))$, by Lemma \ref{common arc}, $cf(\eta(x, a) , t_{\infty}) \subset B( \eta(x, c) , t_{\infty})) \subset B( \eta(x, b) , t_{\infty}))$

 %with $cf(\eta(x, 0), t)$ and $cf(\eta(x, 1), t)$ lying in different components of $\R^2 \setminus cf(\gamma, t)$. 

%Therefore,  $cf(  \eta(x, s), t)$ intersects $cf(\gamma, t)$ at two points,  for any $ t\in (0, T]$ and $s \in [s_1, s_0]$.

%\text{}
%\par \par \par \par \par \par \par \par \par
\par

\begin{lemma}\label{hit z}
Let $\gamma \in \mathfrak{J}_T$, and $z \in \cf(\gamma, t_{\infty})$ for $t_{\infty} \in (0, T]$. 
Then there is a Jordan curve that intersects $\gamma$ at exactly $2$ points, and evolves to pass through $z$ at time $t_{\infty}$, i.e. $z \in \cf(\eta, t_{\infty})$.

\end{lemma}

\begin{proof}

For any $x \in \gamma$, let  $\eta(x) = \eta(x, \nu_0(x))$ be the Jordan curved defined in Lemma \ref{Existence}. Let $p(\gamma, \eta, x, t) \in \cf(\gamma, t) \cap \cf(\eta(x), t)$ be the trajectory starting from $p(\gamma, \eta, x, 0) = x$ as in Lemma \ref{p_T 2}. Since $\eta$ is continuous in Fr\'echet distance, by Lemma \ref{p_T 2}, $p$ is continuous. The map $p$ defines a homotopy from the identity map on $\gamma$ to the map $p(\gamma, \eta, \cdot, t_{\infty}): \gamma \to \cf(\gamma, t_{\infty})$, so $p(\gamma, \eta, \cdot, t_{\infty})$ winds once around 
$\cf(\gamma, t_{\infty})$, which implies there is some $x \in \gamma$ such that $p(\gamma, \eta, x, t_{\infty}) = z$. Therefore, $ \eta(x)$ has the desired property. 
\end{proof}
%\text{}

\begin{proof}[Proof of Theorem 1.1] This proof is analogous to  Michael Dobbins's proof of Theorem 1.1 in \cite{Dobbins}.  For the reader's convenience, the proof is shown here.

\par

Let $\gamma_k$ satisfy the hypothesis of the theorem, and assume
for the sake of contradiction that $\cf(\gamma_k, t_k) \not \to \cf(\gamma_{\infty}, t_{\infty}) $ in Fr\'echet distance. By Theorem 1.2 (1) and Corollary 1.6
of Lauer \cite{Lauer}, $\cf(\gamma_k, t_k)$ is smooth and has length bounded by some constant for all $k$ sufficiently
large, provided that $t_{\infty}> 0$. Therefore, the sequence of constant speed
parameterizations $\psi_k : \sph^1
 \to \cf(\gamma_k, t_k)$ is uniformly equicontinuous, so by the Arzel\`a-Ascoli
theorem, we can restrict to a sequence that converges uniformly to a map $\psi_{\infty}$. Moreover, by
Lemma \ref{J haus}, the range of the limit $\psi_{\infty}$ is $\cf(\gamma_{\infty}, t_{\infty})$. We can also let $\omega : \mathbf{S}^1
\to \cf(\gamma_{\infty}, t_{\infty})$ be
a constant speed parameterization. Let $\varphi : [0, 2\pi]_{\R} \to \R$ by $\varphi(\theta) = -i log(\omega^{-1} \circ \varphi_{\infty}(e^{i\theta}))$, which is just $\omega^{-1} \circ \varphi_{\infty}$ lifted by the standard parameterization of the circle by angle.
Since the map $(\theta \mapsto \omega^{-1}
\circ \varphi_{\infty}(e^{i \theta}))$ is periodic, $\varphi(2\pi) - \varphi(0)$ is a multiple of $2\pi$. We may
choose the direction of $\omega$ so that $\varphi(2\pi) -\varphi(0) \geq 0$. If we had $\varphi(2\pi) - \varphi(0) = 0$, then the
region on one side of $\cf(\gamma_k, t_k) $ would converge to a subset of $\cf(\gamma_{\infty}, t_{\infty})$, but $\gamma_k$ continues
to be a Jordan curve as it evolves, so that cannot happen. Hence, $\varphi(2\pi) - \varphi(0) \neq  0$. If we had
$\varphi(2\pi) - \varphi(0) > 2\pi$, then $\cf(\gamma_k, t_k)$ would wind more than once around a tubular neighborhood
of $\cf(\gamma_{\infty}, t_{\infty})$, which is impossible for a simple closed curve, so $\varphi(2\pi) - \varphi(0) = 2\pi$.

\par

If $\varphi$ were weakly increasing, then $\varphi$ would be the limit of some sequence of strictly increasing functions $\varphi_k \to \varphi$, and we would have homeomorphisms $\tau_k : \mathbf{S}^1 \to \cf(\gamma_{\infty}, t_{\infty})$ given by
$\tau_k(x) = \omega(e^{
i\varphi_k(-i log(x))})) $ that converge to $\psi_{\infty}$, but then the Fr\'echet distance between $\cf(\gamma_k, t_k)$
and $\cf(\gamma_{\infty}, t_{\infty})$ would be bounded by $\sup_{x \in \mathbf{S}^1} || \psi_k(x) - \tau_k(x)|| \to 0$, which contradicts our
assumption that $\cf(\gamma_k, t_k) \not \to \cf(\gamma_{\infty}, t_{\infty})$ in Fr\'echet distance. Hence, $\varphi$ must decrease somewhere, i.e. there is $w < y$ such that $\varphi(y) < \varphi(w)$, and we may choose the period of $\varphi$ so that
$\varphi(w)$ is between $\varphi(0)$ and $\varphi(2\pi) = \varphi(0) +  2\pi$, and may choose $w, y$ arbitrarily close together.
Then, there are $v < w < x < y < z$ where $\varphi(v) = \varphi(y) < \varphi(x) < \varphi(w) = \varphi(z)$.

\par

Let $v_k = \psi_k(e^{iv}), w_k = \psi_k(e^{iw}), x_k = \psi_k(e^{ix}), y_k = \psi_k(e^{iy}), z_k = \psi_k(e^{iz}).$ The situation
so far is that as we traverse $cf(\gamma_k, t_k)$, we pass through $v_k, w_k, x_k, y_k, z_k$ in that order, and
these points respectively converge to $y_{\infty}, w_{\infty}, x_{\infty}, y_{\infty}, w_{\infty}$; see Figure \ref{6}. 
By Lemma \ref{hit z},
there is a Jordan curve $\eta$ that intersects $\gamma$ at exactly $2$ points and such that $\cf(\eta, t_{\infty})$ intersects
$\cf(\gamma_{\infty}, t_{\infty})$ at $x_{\infty}$. Also, we can choose $w, y$ sufficiently close together so that the $2$ points
$\cf(\eta, t_{\infty}) \cap \cf(\gamma_{\infty}, t_{\infty})$ are not on the same arc from $w_{\infty}$ to $y_{\infty}$. That is, $w_{\infty}$ and $y_{\infty}$ are on
opposite sides of $\cf(\eta, t_{\infty})$.

\par

Since level-set flow gives a solution to the curvature flow problem, we can find a nested
sequence of smooth annuli $\{A_i(0)\}_{i=1}^{\infty}$ such that 
$A_1(0) \supset A_2(0) \supset \cdot \cdot \cdot$ such that $\cf(\eta, t) = \bigcap_{i=1}^{\infty} Ai(t)$ where the
boundary of $A_i(t)$ evolves by curvature flow starting $A_i(0)$.
Choose one of the
annuli $A = A_i$ such that $A(t)$ remains an annulus up to time $s \in (t_{\infty}, T)$ and close enough
to $\eta$ that $A(t_{\infty})$ does not contain $w_{\infty}$ or $y_{\infty}$; see Figure \ref{6}. Let $\alpha$ and $\beta$ be the curves on
the boundary of $A(0)$. Note that $\cf(\alpha, t) $ and $\cf(\beta, t)$ are on either side of $\cf(\eta, t)$ for all
$t \in  [0, s]_{\R}$. Hence, each arc of $\gamma_{\infty}$ subdivided by $\alpha \cup \beta$ that goes from a point on $\alpha$ to a
point on $\beta$ must cross $\eta$, and since $\gamma_{\infty}$ meets $\eta$ at only $2$ points, there can be at most $2$ such
arcs.

\par

Since $\alpha, \eta, \beta$ are compact and pairwise disjoint, we may choose $\varepsilon_1 > 0$ small enough that
$\alpha \oplus \varepsilon_1$, $\zeta \oplus \varepsilon_1$, $\beta \oplus \varepsilon_1$ are pairwise disjoint. Let us choose $k$ sufficiently large that $\gamma_k$ is at
most Fr\'echet distance $\varepsilon_1$ from $\gamma_{\infty}$, and $v_k$, $w_k$, $y_k$, and $z_k$ are each outside of $A(t_k)$. To see
that the later condition can be satisfied, recall that these points each approach either $w_{\infty}$ or
$y_{\infty}$, which are bounded away from $A(t_{\infty})$, and $A(t_k) \to A(t_{\infty})$ in Hausdorff distance.

\par

Consider an arc $\xi$ of $\gamma_k \cap A(0)$ from $\alpha$ to $\beta$. Since $\gamma_k$ is within Fr\'echet distance $\varepsilon_1$ from $\gamma_{\infty}$
there is some map $\psi : \xi \to \gamma_{\infty}$ such that $||\psi(x) - x_k|| \leq  \varepsilon_1$ for all $x \in \xi$, so $\psi(\xi)$ must be an
arc of $\gamma_{\infty}$ from $\alpha \oplus \varepsilon_1$ to $\beta \oplus \varepsilon_1$. Hence, $\psi(\xi)$ must intersect $\zeta$, and since $\gamma_{\infty}$ only intersects
$\zeta$ at $2$ points, there are only $2$ arcs of $\gamma_k \cap A(0)$ from one boundary of $A(0)$ to the other.
Therefore, we can find a curve $\eta$ of area $0$ that winds once around the interior of $A(0)$ and
intersects $\gamma_k$ at only $2$ points; see Figure \ref{6}.

\begin{figure}
\centering
\includegraphics[width = 13cm]{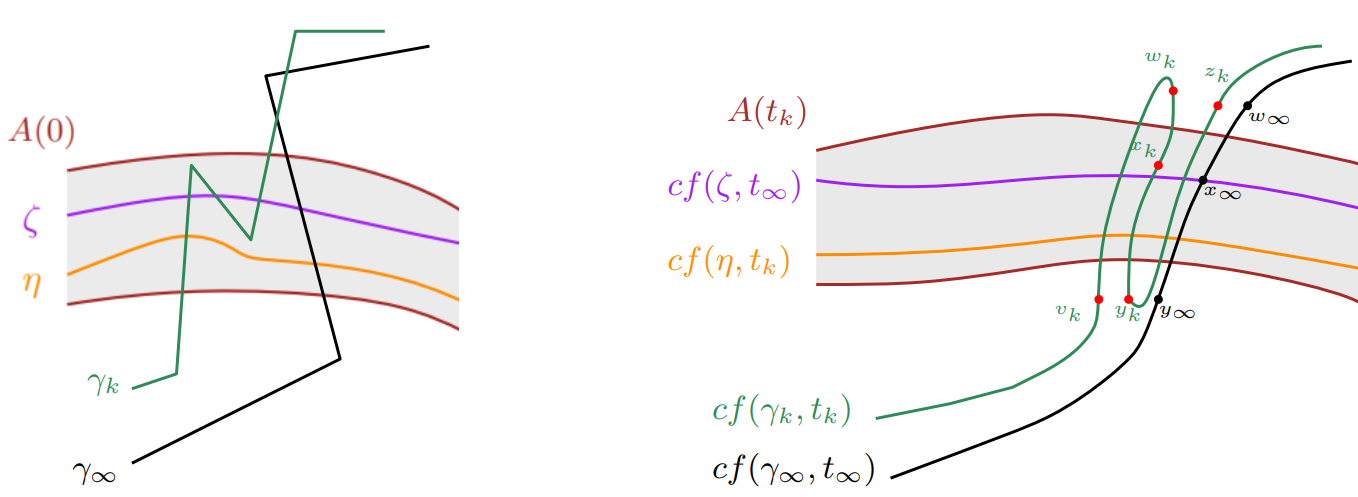}
\caption{ Assuming $\cf(\gamma_k, t_k) \to \cf(\gamma_{\infty}, t_{\infty})$ in Hausdorff distance but not
in Fr\'echet distance, $\gamma_k$ evolves to pass close by a point $x_{\infty}$three times, and
the curve $\zeta$ only intersects $\gamma_{\infty}$ twice and evolves to pass through $x_{\infty}$. The
curve $\eta$ crosses $\gamma_k$ only twice, but evolves to cross $\cf(\gamma_k, t_k)$ more than twice,
contradicting Angenent’s theorem.
  }
\label{6}
\end{figure}

By Angenent’s theorem, $\cf(\eta, t)$ remains disjoint from $\cf(\alpha, t)$ and $\cf(\beta, t)$, and crosses $\cf(\gamma_k, t)$
at most twice, for as long as a solution to curvature flow exists \cite{Sigurd}, and by Lauer’s theorem,
the evolving curves $\cf(\alpha, t), \cf(\beta, t)$ on either side of $\eta$ ensures that a unique solution exists
for $\cf(\eta, t)$ up to time  $s$ \cite{Lauer}.

\par

Since $\cf(\gamma_k, t_k)$ passes though $v_k, w_k, y_k, z_k$ in that order, and $v_k$ and $y_k$ are on one side of
$A(t_k)$, and $w_k$ and $z_k$ are on the other side of $A(t_k)$, $\cf(\gamma_k, t_k)$ must cross from one side
of $A(t_k)$ to the other at least $4$ times. Since $\cf(\eta, t)$ never intersects the boundary of $A(t)$,
$\cf(\eta, t_k)$ winds once around $A(t_k)$. Therefore, $\cf(\eta, t_k)$ intersects $\cf(\gamma_k, tk)$ at least at $4$ points,
but $\eta$ intersects $\gamma_k$ at only $2$ points, which contradicts Angenent’s theorem that the number
of intersection points does not increase; see Figure \ref{6}. Thus, our assumption that $\gamma_k \not \to  \gamma_{\infty}$
in Fr\'echet distance cannot hold.  
\end{proof}

\section{Conclusions}
\noindent
We  showed that the evolution of a Jordan curve on the plane by curvature flow depends continuously on the initial curve in Fr\'echet distance. 
%Fr\'echet topology offers an apt setting for scrutinizing the limit behavior of evolving geometries. Studying curvature flow in Fr\'echet topology 
The theoretical results we obtain in this work
helps to connect this geometric evolution to the broader framework of functional analysis, providing a deeper understanding of the underlying mathematical structures. The proof of Theorem \ref{1.1} relies on Joseph Lauer's theorem that a nonsmooth curve in $\R^2$  instantly becomes smooth and has length bounded by some constant \cite{Lauer}. Therefore, it remains an open question whether our results can be extended to Jordan curves on other surfaces. Another promising research direction can be the study of the continuity of surfaces in $\R^3$, where the challenge in establishing continuity arises from the potential formation of  singularities, such as the neckpinch singularity.

%can be generalized to Jordan curve on smooth surfaces that are hoemomorphic to $\R^2$. 


\begin{thebibliography}{}


     

\bibitem{Sigurd}
{\sc Sigurd Angenent} {\em Parabolic Equations for Curves on Surfaces: Part II. Intersections, Blow-up and Generalized Solutions.} Annals of Mathematics 133, no. 1 (1991): 171–215. 

\bibitem{Dobbins}
{\sc Michael Gene Dobbins.} {\em Continuous Dependence of Curvature Flow on Initial Conditions.} arXiv:2106.08907, 2021.



\bibitem{Grayson}
{ \sc Matthew A. Grayson.} {\em The heat equation shrinks embedded plane curves to round points.} J. Differential Geom. 26(2): 285-314 (1987).




\bibitem{Gage}
{ \sc M. Gage \& R.S. Hamilton} {\em The Heat Equation Shrinking Convex Plane Curves.} J. Differential Geom. 23(1): 69-96 (1986).



%\bibitem{MG}
%{\sc Michael E Gage.} {\em Curve shortening on surfaces.} Annales scientifiques de l’Ecole normale
%sup´erieure, 23(2):229–256, 1990.
\bibitem{Lauer}
{\sc Joseph Lauer.} {\em A New Length Estimate for Curve Shortening Flow and Low Regularity Initial Data.} Geom. Funct. Anal. 23, 1934–1961 (2013). 






%\bibitem{Pom}
%{\sc Christian Pommerenke.} {\em Boundary Behavior of Conformal Maps.} Springer, 1992.
\end{thebibliography}
\end{document}